\documentclass[10pt,a4paper]{article}
\usepackage[numbers,sort&compress]{natbib}
\usepackage[T1]{fontenc}
\usepackage[utf8]{inputenc}
\usepackage{authblk}
\usepackage{amsmath,amssymb,amsthm,esint,bm}
\usepackage{mathrsfs}
\usepackage{bookmark}
\usepackage{amsmath}
\usepackage{enumitem}
\allowdisplaybreaks[3]

\newtheorem{definition}{Definition}[section]
\newtheorem{theorem}[definition]{Theorem}
\newtheorem{lemma}[definition]{Lemma}
\newtheorem{proposition}[definition]{Proposition}
\newtheorem{corollary}[definition]{Corollary}
\theoremstyle{remark}
\newtheorem{remark}[definition]{Remark}
\numberwithin{equation}{section}

\newcommand{\abs}[1]{\lvert#1\rvert}
\newcommand{\Abs}[1]{\left\lvert#1\right\rvert}
\newcommand{\norm}[1]{\lVert#1\rVert}

\newcommand{\rn}{{\mathbb{R}^d}}

\allowdisplaybreaks[4]

\title{Sharp gradient estimates for a class of singular/degenerate fully nonlinear elliptic equations with oblique boundary conditions and Hamiltonian terms}

\author[a]{Wentao Huo}
\affil[a]{School of Mathematical Sciences, Nankai University, Tianjin 300071, P.R. China}
\date{\today}

\usepackage{hyperref}
\begin{document}
	\maketitle
	\footnotetext[1]{E-mail: huowentaoouc@163.com (W. Huo)} 

\begin{abstract}
	This paper focus on a class of singular or degenerate fully nonlinear elliptic equations with Hamiltonian terms under oblique boundary conditions on $C^{1}$ domains. Under quite general conditions on the singularity/degeneracy of the model, we establish the sharp $C^{1,\alpha}$ regularity up to the boundary within a unified framework. 
	
	
\end{abstract}


\section{Introduction}\label{section1}
Since the 1980s, the notion of the viscosity solution has been applied widely in the study of non-divergent equations especially of fully nonlinear elliptic equations, and a series of important interior regularity results and global results with Dirichlet/Neumann/Oblique boundary conditions have been obtained (see \cite{Caff1,Caffarelli1989,LiZhang,Silvestre2006CPDE,Silvestre2014CPDE,Winter} and references therein) . 

Despite of the profound importance of the aforementioned works, a large number of mathematical models involve operators whose diffusion properties (ellipticity) may degenerate or blow-up along an a priori unknown region, such as $S(u):=\left\{x:Du(x)=0\right\}$, which depends on the solution itself. From the mathematical view point, this fact impels less efficient diffusion features for the model near such a region and therefore the regularity theory for solutions to such equations become more sophisticated. 

In this paper, we examine the gradient regularity of viscosity solutions to the following degenerate/singular fully nonlinear elliptic equations with Hamiltonian terms and oblique boundary condition
\begin{equation}\label{zhumodel}
	\left\{
	\begin{array}{rcrcl}
		\Phi(\abs{Du}, x)F(D^2 u, x)+\mathcal{H}(Du,x)
		 &=&f(x) & \text{in} & \Omega, \\
		\mathcal{B}(Du,u,x) & = & g(x) & \text{on} & \partial \Omega,
	\end{array}
	\right.
\end{equation}
where 
\begin{equation}\label{1100}
	\mathcal{B}(Du,u,x):= \beta(x)\cdot Du+\gamma(x)u
\end{equation}
is the oblique boundary condition, $\Omega \subset \mathbb{R}^d$ $(d\geq 2)$ is a bounded $C^{1}$ domain, $F$ is a uniformly elliptic operator, $\Phi:[0,\infty)\times\Omega\rightarrow [0,\infty)$ is a continuous map possessing degeneracy/singularity as the gradient vanishes, the Hamiltonian term $\mathcal{H}(Du,x)$ and data $f,\beta,\gamma,g$ fulfill appropriate conditions (to be clarified soon). We remark that the word oblique means "having a sloping direction". In \eqref{1100}, $\beta$ plays a role of slope. When $\beta$ makes an angle more than some fixed level on $\partial\Omega$, the condition \eqref{1100} is said to be regular and otherwise it is said to be nonregular. By way of illustration, the simplest example of the regular oblique boundary condition is the Neumann boundary condition, that is, the case of $\beta={\mathbf{n}}$ and $\gamma=0$ where ${\mathbf{n}}$ is the outward normal vector of $\partial\Omega$. Therefore, we can view the condition \eqref{1100} as a sort of generalization of the Neumann boundary condition.

The study of equations like \eqref{zhumodel} arises naturally in the optimal stochastic control problems \cite{Fleming,Leoni,Birindelli2019,Lions1989}, in image enhancements \cite{Levine2006SIAM,Bronzi2020}, and as a generalization of Hamilton-Jacobi equations \cite{Armstrong2015,Pimentel2023}. On the other hand, problems involving oblique tangential derivative conditions are naturally motivated by the modeling of several phenomena, such as reflected shocks in transonic flow \cite{CKL00}, the theory of capillary problems \cite{ESY96,FinnLuli}, Brownian motion \cite{EFH,RW}. Over the last two decades, the theory of equations like \eqref{zhumodel} has experienced remarkable progress. 
In the pioneering works \cite{Birindelli2004,Birindelli2006,Birindelli2007CPAA,Birindelli2010JDE}, existence and uniqueness, maximum principles, Harnack inequality and H\"{o}lder regularity were obtained. The graident H\"{o}lder regularity of solutions is the subject of groundbreaking work \cite{Imbert1}, where they established interior $C^{1,\alpha}$ regularity of viscosity solutions to $|Du|^{p}F(D^{2}u)=f$ with $p\geq 0$; the optimality of its H\"{o}lder exponent for the same problem is shown in \cite{Ricarte}. Subsequently, this type of local regularity result has been extended to various kinds of degenerate/singular fully nonlinear equations, please refer to \cite{Bronzi2020,Fili, Silva2020,Fang, Silva2023,Baasandorj,Andrade} and the references therein. It is noteworthy to mention that the recent paper \cite{Huo2026} derived local optimal $C^{1,\alpha}$ regularity of solutions to \eqref{zhumodel}.

As for the Dirichlet problem in the purely degenerate/singular setting, i.e., $\mathcal{H}\equiv0$, there has been much progress in the regularity of solutions. For instance, if $\Phi(t,x)=t^{p}$ with $p\geq 0$, Araújo and Sirakov \cite{Arau} obtained sharp boundary and global $C^{1,\alpha}$ regularity in the presence of sufficiently regular domain and boundary datum. More or less simultaneously, Bezerra J$\acute{\rm u}$nior et al. \cite{Silva2023} considered the variable-exponent nonhomogeneous scenario, i.e., $\Phi(t,x)=t^{p(x)}+a(x)t^{q(x)}$ with $0\leq p(\cdot)\leq q(\cdot)$, and proved that solutions are $C^{1,\alpha}$ regularity up to the boundary for some $\alpha>0$ depending additionally on the regularity of the boundary $\partial\Omega$, as well as on the boundary data. Afterwards, these results were extended to the Dirichlet problem exhibiting more general degeneracy/singularity in \cite{Byun1}. On the other hand, with regard to the Dirichlet problem for the singular or degenerate fully nonlinear equations involving  Hamiltonian terms, Birindelli-Demengel's seminal works \cite{Birindelli2014ESAIM, Birindelli2015, B-Demengel2016, B-Demengel2019} investigated the Dirichlet problem \eqref{zhumodel} with $\Phi(t,x)=t^{p}$, $\mathcal{H}(t,x)=h(x)|t|^{m}$ and $\mathcal{B}=u$ on a $C^{2}$ domain. Based on an improvement-of-flatness approach, they obtained global $C^{1,\alpha}$ regularity for some unspecified exponent $\alpha\in(0,1)$ in \cite{Birindelli2014ESAIM, Birindelli2015, B-Demengel2016} for the sublinear/linear cases $0<m\leq p+1$, and in \cite{B-Demengel2019} for the superlinear/subquadratic cases $p+1<m\leq p+2$.

Within the framework of Neumann boundary conditions, Patrizi \cite{Patrizi2008JMPA} proved a global $C^{1,\alpha}$ regularity for the singular case under the homogeneous Neumann condition and on a $C^{2}$ domain. Subsequently, Banerjee and Verma \cite{BanerjeeVerma} considered the Neumann problem \eqref{zhumodel} with $\Phi(t,x)=t^{p}$ ($p\geq 0$), $\mathcal{H}=0$, $\beta={\mathbf{n}}$ and $\gamma=0$, and established $C^{1,\alpha}$ regularity up to the boundary on a $C^{2}$ domain. In this setting, the optimal gradient H\"{o}lder regularity was derived in \cite{Ricarte2020Neumann}. 

In the context of oblique boundary conditions, Byun et al. \cite{Byun2025CVPDE} investigated \eqref{zhumodel} with $\Phi(t,x)=t^{p}$ ($p\geq 0$) and $\mathcal{H}=0$. Based on perturbation and compactness arguments, they proved sharp $C^{1,\alpha}$ estimate up to the boundary. Concerning this problem, Bessa et al. \cite{BesRicSil26} extended this result for the optimal gradient regularity theory to the degenerate oblique setting under relaxed convexity assumptions. Afterwards, Byun et al. \cite{BKK26arxiv} studied the problem \eqref{zhumodel} with $\mathcal{H}=0$ and $\Phi$ exhibiting general degeneracy/singularity (see condition \eqref{A2} described 
below), and established global $C^{1,\alpha}$ regularity. However, with regard to this type of regularity result for singular or degenerate fully nonlinear equations with Hamiltonian terms under oblique boundary conditions, there are very few results. To the best of our knowledge, the only study in this direction was was carried out by \cite{Ricarte2026arxiv}, which extended the estimates to the following degenerate fully nonlinear elliptic equation with Hamiltonian terms
\begin{equation*}
	\left\{
	\begin{array}{rcrcl}
		|Du|^{p}F(D^2 u)+h(x)|Du|^{m}
		&=&f(x) & \text{in} & \Omega, \\
		\beta(x)\cdot Du+\gamma(x)u & = & g(x) & \text{on} & \partial \Omega,
	\end{array}
	\right.
\end{equation*}
with $p\geq 0$ and $0<m\leq 1+p$.

Inspired by the work mentioned above, we naturally consider degenerate/singular fully nonlinear elliptic equation with variable coefficient and Hamiltonian terms under oblique boundary condition. Our primary focus is on establishing sharp global $C^{1,\alpha}$ regularity estimate for viscosity solutions to \eqref{zhumodel} in a unified way, provided that $\Phi(|Du|,\cdot)$ satisfies the quite general degeneracy/singularity \eqref{A2} below.
\subsection{Assumptions and main results}
We begin by introducing some notations here.
\begin{itemize}
	\item  $S^{d}$ denotes the set of all real $d\times d$ symmetric matrices.
	\item  For $r>0$, $B_{r}(x_{0})$ denotes the open ball with radius $r$ and centred at $x_{0}\in\rn$. In particular, we shall simply denote $B_{r}:=B_{r}(0)$. 
	\item [$\bullet$] We write ${B}^+_r := {B}_r \cap \{x_d>0\}$ and ${T}_r := {B}_r \cap \{x_d=0\}$.
	\item  For a domain $\Omega \subset \mathbb{R}^d$ and $x_0 \in \mathbb{R}^d$ we write $\Omega_r := \Omega \cap {B}_r$, $\partial \Omega_r := \partial \Omega \cap {B}_r$ and $\Omega_r(x_0):=\Omega \cap {B}_r(x_0)$, $\partial \Omega_r(x_0):=\partial \Omega \cap {B}_r(x_0)$.
	 \item For any measurable set $A$ with $|A| =0$ and measurable function $\phi$, we write
	 $$\fint_{A}\phi{\rm d}x:=\frac{1}{|A|}\int_{A}\phi{\rm d}x.$$	
\end{itemize}

We proceed by presenting the assumptions to be employed throughout this article.
\begin{enumerate}[label=(\text{\bf A}\arabic{enumi}),ref=\textbf{A}\arabic{enumi}]
	\item \label{A1} \textbf{({Uniformly ellipticity})}
	The fully nonlinear operator $F:S^{d}\times \overline{\Omega}\rightarrow \mathbb{R}$ is uniformly $(\lambda,\Lambda)$-elliptic, that is, there exist constants $0<\lambda\leq \Lambda$ such that, for any $M,N\in S^{d}$ and $x\in \overline{\Omega}$, 
	\begin{equation*}
		\mathcal{P}_{\lambda,\Lambda}^{-}(N)\leq F(M+N,x)-F(M,x)\leq \mathcal{P}_{\lambda,\Lambda}^{+}(N).
	\end{equation*}
	Here, $P_{\lambda,\Lambda}^{\pm}$ are the usual Pucci extremal operators, defined for each $M\in S^{d}$ by
	$$\mathcal{P}_{\lambda,\Lambda}^{+}(M):=\Lambda\sum\limits_{e_{i}>0}e_{i}+\lambda\sum\limits_{e_{i}<0}e_{i},\quad \mathcal{P}_{\lambda,\Lambda}^{-}(M):=\lambda\sum\limits_{e_{i}>0}e_{i}+\Lambda\sum\limits_{e_{i}<0}e_{i},$$
	where $\{e_{i}\}_{i=1}^{d}$ are the eigenvalues of the matrix $M$. For normalization purposes, we assume
	that $F(0,\cdot) = 0$.
	\item \label{A2} \textbf{(Singularity/Degeneracy)} The function $\Phi:[0,\infty)\times\Omega\rightarrow [0,\infty)$ is a continuous map satisfying the following properties:\\
	(i) there exist constants $s(\Phi)\geq i(\Phi)>-1$ such that for every $x\in \Omega$, the map $t\mapsto \frac{\Phi(t,x)}{t^{i(\Phi)}}$ is almost non-decreasing and the map $t\mapsto \frac{\Phi(t,x)}{t^{i(\Phi)}}$ is almost non-increasing  with constant $L\geq 1$ in $(0,\infty)$;\\
	(ii) there exist constants $0<\nu_{0}\leq \nu_{1}$ such that $\nu_{0}\leq \Phi(1,x)\leq \nu_{1}$ for all $x\in \Omega$.
	\item \label{A3} \textbf{(Assumption on Hamiltonian term)} The function $\mathcal{H}: \rn \times \Omega\rightarrow \mathbb{R}$ is continuous and there exist constants $\mathcal{K},\mathcal{M}>0$ and $0<m\leq 1+i(\Phi)$ such
	that
	\begin{equation*}
		|\mathcal{H}(t,x)|\leq \mathcal{K}+\mathcal{M}|t|^{m}
	\end{equation*}
	for every $t\in\rn$ and $x\in\Omega$.
	\item \label{A4} \textbf{(Assumptions on data)} The source term $f\in C({\Omega}) \cap L^{\infty}(\Omega)$, $\beta\in C^{\alpha^{\prime}}(\partial\Omega;\rn)$
	and $\gamma,g\in C^{\alpha^{\prime}}(\partial\Omega)$ for some $\alpha^{\prime}\in(0,1)$. In addition, there exists a constant $\mu_{0}>0$ such that
	\begin{equation}\label{1112}
		\beta\cdot {\mathbf{n}}\geq \mu_{0}\quad\text{and}\quad \|\beta\|_{L^{\infty}(\partial\Omega)}\leq 1,
	\end{equation}
	where ${\mathbf{n}}$ denotes the inner unit normal vector to $\partial \Omega$.
	\item \label{A5} \textbf{(Assumptions on domain)} $\Omega\subset \rn$ is a bounded $C^{1}$ domain.
\end{enumerate}
	
	Here a function $\varrho:(0,\infty)\rightarrow \mathbb{R}$ is almost non-decreasing or almost non-increasing with constant $L\geq 1$ if
	$\varrho(s)\leq L\varrho(t)$ or $\varrho(t)\leq L\varrho(s)$ respectively for $0 < s \leq  t$.

Since the model in \eqref{zhumodel} depends on $x$, it is necessary to introduce the function ${\rm osc}_{F}$ in order to 
measure the oscillation of $F$ in $x_{0}$. To this end, we define
\begin{equation*}
	\overline{\Omega}\ni 
	x,x_{0}\mapsto{\rm osc}_{F}(x,x_{0}):=\sup\limits_{M\in S^{d}\setminus \{0\}}\frac{\Abs{F(M,x)-F(M,x_{0})}}{\|M\|}.
\end{equation*}
For simplicity purposes, denote ${\rm osc}_{F}(x):={\rm osc}_{F}(x,0)$.

We emphasize that the singularity/degeneracy function $\Phi(\cdot,\cdot)$ considered in \eqref{A2} is quite general, which encompasses a broad class of growth conditions. Some important examples are given by
\begin{itemize}	
	\item  [{\rm (1)}] {\bf $p$-growth:} $\Phi(t,x)=t^{p}$ with $p>-1$;
	\item  [{\rm (2)}] {\bf Variable $p(x)$-growth:} $\Phi(t,x)=t^{p(x)}$ with $p(x)>-1$;
	\item  [{\rm (3)}] {\bf Double phase growth:} $\Phi(t,x)=t^{p}+a(x)t^{q}$ with $-1<p\leq q$ and $0\leq a(\cdot)\in C(\Omega)$;
	\item  [{\rm (4)}] {\bf Variable exponent double phase:} $\Phi(t,x)=t^{p(x)}+a(x)t^{q(x)}$ with $-1<p(x)\leq q(x)$ and $0\leq a(\cdot)\in C(\Omega)$;
	\item  [{\rm (5)}] {\bf Borderline double phase case:} $\Phi(t,x)=t^{p}+t^{q}\log(t+1)$ with $-1<p\leq q$ and $0\leq a(\cdot)\in C(\Omega)$;
	\item  [{\rm (6)}] {\bf Multi-phase growth:} $\Phi(t,x)=t^{p}+a(x)t^{q}+b(x)t^{\gamma}$ with $-1<p\leq q,\gamma$ and $0\leq a(\cdot),b(\cdot)\in C(\Omega)$.
\end{itemize}
Furthermore, condition \eqref{A3} on Hamiltonian term includes some typical examples, such as
\begin{itemize}	
	\item  [{\rm$\bullet$}] $\mathcal{H}(t,x)=h(x)|t|^{m}$ for $m>0$ and function $h\in C(\Omega)\cap L^{\infty}(\Omega)$;
	\item  [{\rm$\bullet$}] $\mathcal{H}(t,x)=\left\langle h(x),t \right \rangle |t|^{m-1}$ for $m>0$ and the vector field $h\in C(\Omega,\rn)\cap L^{\infty}(\Omega,\rn)$;
	\item  [{\rm$\bullet$}] $\mathcal{H}(t,x)=\sum\limits_{i=1}^{N}h_{i}(x)|t|^{m_{i}}$ for $m_{i}>0$, function $h_{i}\in C(\Omega)\cap L^{\infty}(\Omega)$, $i=1,2,\cdots,N$, and $m=\max\left\{m_{1},m_{2},\cdots,m_{N}\right\}.$
\end{itemize}

Finally, the assumption \eqref{A5} was motivated by the approach developed in \cite{Pimentel2023}. More precisely, we assume that $0\in \partial\Omega_{1}$ and that the boundary $\partial\Omega_{1}$ can be locally described as the graph of a $C^{1}$ function $\psi=\psi_{\Omega}:{T}_{1}\to\mathbb{R}$. Namely,
\begin{equation*}
	\partial\Omega_{1}=\{(x',x_{d})\in {B}_{1}\,:\,x_{d}=\psi(x')\}
	\qquad \text{and}\qquad 
	\{(x',x_{d})\in {B}_{1}\,:\,x_{d}>\psi(x')\}\subset\Omega_{1}.
\end{equation*}
We also write  $\beta\in C^{\alpha^{\prime}}({T}_{1})$ and identify $\beta(x)$ with $\beta(x',\psi(x^{\prime}))$.

We are now in a position to state the global $C^{1,\alpha}$ regularity estimate of solutions to the Dirichlet problem \eqref{zhumodel}.
\begin{theorem}\label{thm1}
	Let $\Omega \subset \mathbb{R}^d$ is a bounded $C^{1}$ domain with $0\in \partial \Omega_{1}$.
	Let $u$ be a bounded viscosity solution of 
	\begin{equation}\label{111zhumodel}
		\left\{
		\begin{array}{rcrcl}
			\Phi(\abs{Du}, x)F(D^2 u, x)+\mathcal{H}(Du,x)
			&=&f(x) & \text{in} & \Omega_{1}, \\
			\mathcal{B}(Du,u,x) & = & g(x) & \text{on} & \partial \Omega_{1},
		\end{array}
		\right.
	\end{equation}
 under the assumptions \eqref{A1}-\eqref{A4}. Then, there exists a universal constant $\eta_{0}>0$ such that if
 $$\left(\fint_{B_{r}(x_{0})\cap \Omega_{1}}\Abs{{\rm osc}_{{F}}(x,x_{0})}^{d}\text{d}x\right)^{1/d}\leq \eta_{0}$$
 for all $x_{0}\in  \Omega_{1/2}\cup\partial \Omega_{1/2}$ and $r\in \left(0,\frac{1}{2}\right]$, then $u\in C^{1,\alpha}(\overline{\Omega_{1/2}})$ with
	\begin{equation}\label{exponent}
		\alpha\in \left\{
		\begin{array}{lcl}
			(0,\alpha_{0})\cap(0,\alpha^{\prime})\cap \left(0,\frac{1}{1+s(\Phi)}\right] & \text{if} & i(\Phi)\geq 0, \\
			(0,\alpha_{0})\cap(0,\alpha^{\prime})\cap \left(0,\frac{1}{1+s(\Phi)-i(\Phi)}\right] &\text{if}&  -1<i(\Phi)<0.
		\end{array}
		\right.
	\end{equation}
In addition, the following estimates hold:
	\begin{itemize}
		\item [{\rm$({{\rm i}})$}] if $m<1+i(\Phi)$, then
		\begin{equation*}
			\|u\|_{C^{1, \alpha}({\overline{\Omega_{1/2}}})}\leq C\left(1+\|u\|_{L^{\infty}\left(\Omega_{1}\right)}+\|g\|_{C^{\alpha^{\prime}}\left(\partial\Omega_{1}\right)}+\left(\|f\|_{L^{\infty}\left(\Omega_{1}\right)+\mathcal{K}}\right)^{\frac{1}{1+i(\Phi)}}+\mathcal{M}^{\frac{1}{1+i(\Phi)-m}}\right),
		\end{equation*}
		where the constant $C$ depends on $d,\lambda,\Lambda,\alpha,m,L,i(\Phi),\mu_{0},\nu_{0},[\beta]_{C^{\alpha^{\prime}}(\partial\Omega_{1})},\|\gamma\|_{C^{\alpha^{\prime}}\left(\partial\Omega_{1}\right)}$ and $C^1$ modulus of $\partial \Omega_1$;
		\item [{\rm$({{\rm ii}})$}] if $m=1+i(\Phi)$, then
		\begin{equation*}
			\|u\|_{C^{1, \alpha}({\overline{\Omega_{1/2}}})}\leq C\left(1+\|u\|_{L^{\infty}\left(\Omega_{1}\right)}+\|g\|_{C^{\alpha^{\prime}}\left(\partial\Omega_{1}\right)}+\left(\|f\|_{L^{\infty}\left(\Omega_{1}\right)+\mathcal{K}}\right)^{\frac{1}{1+i(\Phi)}}\right),
		\end{equation*}
		where the constant $C$ depends in addition on $\mathcal{M}$.
	\end{itemize}
\end{theorem}
\begin{remark}\label{optimal exponent}
	The constant $\alpha_0 = \alpha_0(n,\lambda,\Lambda,\delta_0) \in (0,1]$ in the statement of the theorem denotes the optimal exponent of regularity theory for a homogeneous equation
	with a constant oblique boundary condition. More precisely, any viscosity
	solution $v$ of
	\begin{equation*}
		\left\{\begin{aligned}
			F(D^{2}v) &=0 &&\text{in } B_{1}^{+}, \\
			\beta_0\cdot Dv&=0 &&\text{on } T_{1},
		\end{aligned}\right.
	\end{equation*}
	where $\beta_0$ is a constant vector satisfying
	\eqref{1112}, belongs to the class $C^{1,\alpha_0}_{\text{loc}}(B^+_1)$ for some $\alpha_{0}\in(0,1]$ and satisfies the estimate
	\begin{align*}
		\norm{v}_{C^{1,\alpha_0}(B^+_{3/4})} \leq
		C_e\norm{v}_{L^\infty(B^+_1)},
	\end{align*}
	for some constant $C_e = C_e(n,\lambda,\Lambda,\delta_0)>1$. More precisely,	Li and Zhang \cite[Theorems 1.2 and 1.3]{LiZhang} proved that there exists a universal constant $\alpha_0\in (0,1)$ for any uniformly elliptic $F$ and that $\alpha_0=1$ for any convex/concave operator $F$. In addition, Bessa, Ricarte, and Silva \cite[Proposition 5.1]{BesRicSil26} showed that $\alpha_{0}=1$ also holds for a  quasiconvex/quasiconcave operator $F$.
\end{remark}
\begin{remark}
Due to the generality of degeneracy/singularity and Hamiltonian terms, our main result Theorem \ref{thm1} covers the regularity results previously obtained in \cite{BKK26arxiv,BesRicSil26,BanerjeeVerma,Byun2025CVPDE,Ricarte2026arxiv}. Notably, our finding is new even for the singular case.
\end{remark}

An immediate consequence of Theorem \ref{thm1} is following optimal regularity result under the quasiconvexity/quasiconcavity regime of the governing operator. 
\begin{corollary}\label{quanjicoro1}
	Under the assumptions of Theorem \ref{thm1}, suppose further that operator $F$ is quasiconvex (or quasiconcave). Then  
	\begin{itemize}
		\item [{\rm$({{\rm i}})$}] if $i(\Phi)\geq 0$, then $u\in C^{1,\frac{1}{1+s(\Phi)}}(\overline{\Omega_{1/2}})$;
		\item [{\rm$({{\rm ii}})$}] if  $-1<i(\Phi)<0$, then $u\in C^{1,\frac{1}{1+s(\Phi)-i(\Phi)}}(\overline{\Omega_{1/2}})$.
	\end{itemize}
\end{corollary}

The remainder of this paper is organized as follows. In Section \ref{2asection2}, we introduce the notions and some preliminaries. In Section \ref{sec3333}, we establish the pointwise boundary $C^{1,\tilde{\alpha}}$ estimate and complete the proof of Theorem \ref{thm1} regarding global H\"{o}lder regularity of the gradient.
\section{Preliminaries}\label{2asection2}
In this section, we briefly present some background knowledge for our discussion. In the entire paper, the symbol $C$ denotes a positive constant whose value may vary from line to line, and only the relevant dependencies are specified in parentheses. Besides, a constant is said to be universal if it depends at most upon the structure constants in \eqref{A1}-\eqref{A5}.

We start off by denoting
 \begin{equation*}
	\mathcal{L}_{\lambda,\Lambda}^{+}(D^{2}u,Du):=\mathcal{P}_{\lambda,\Lambda}^{+}(D^{2}u)+\Lambda|Du|,
\end{equation*}
 \begin{equation*}
	\mathcal{L}_{\lambda,\Lambda}^{-}(D^{2}u,Du):=\mathcal{P}_{\lambda,\Lambda}^{-}(D^{2}u)-\Lambda|Du|.
\end{equation*}

To proceed, we recall a result on the boundary H\"{o}lder regularity for uniformly elliptic equations on $C^{1}$ domains that hold only where the gradient is far from some point.
\begin{lemma}\cite[Theorem 3.1]{Byun2025CVPDE}\label{lem2.1}
	Let $\Omega$ be a bounded $C^{1}$ domain. Then there exists a small $\iota=\iota(\mu_{0}) \in(0, \mu_{0}/2]$ such that, if $[\psi]_{C^{1}}\leq \mu$, for any $u\in C(\overline{\Omega_{1}})$ satisfying
	\begin{equation*}
		\begin{cases}
			\mathcal{L}^{+}_{\lambda, \Lambda}(D^{2}u,Du)\geq -C_{0} & \text{in }\;\; \left\{|Du-q|>\theta\right\}\cap\Omega_{1}, \\ 
			\mathcal{L}^{-}_{\lambda, \Lambda}(D^{2}u,Du) \leq C_{0} & \text{in }\;\; \left\{|Du-q|>\theta\right\}\cap\Omega_{1}, \\
			\beta\cdot Du =g &\text{on }\;\; \partial\Omega_{1},\\
			\norm{u}_{L^{\infty}(\Omega_{1})}\leq 1, \\
			\norm{g}_{L^{\infty}(T_{1})}\leq C_{0},
		\end{cases}
	\end{equation*}    
	for some $q\in\mathbb{R}^{n}$ and $0<\theta\leq1$, then $u\in C^{\alpha}(\overline{\Omega_{1/2}})$ for some $\alpha=\alpha(d, \lambda, \Lambda,\mu_{0})\in(0,1)$ with the estimate
	\begin{align*}
		\norm{u}_{C^{\alpha}(\overline{\Omega_{1/2}})}\leq C,
	\end{align*}
	where $C=C(d, \lambda, \Lambda, \mu_{0}, C_{0})>0$ does not depend on $q$.
\end{lemma}

In the sequel, for the reader's convenience, we shall present the notion of viscosity solution for \eqref{zhumodel}. Denote
\begin{equation}\label{model22}
	G(D^{2}u,Du,x):=\Phi(\abs{Du}, x)F(D^2 u, x)+\mathcal{H}(Du,x).
\end{equation}
\begin{definition}
	\label{dingyi}
	A function $u\in LSC(\overline{\Omega})$ (resp. $u\in USC(\overline{\Omega})$) is a viscosity supersolution (resp. subsolution) to 
	\begin{equation}\label{problem1}
		\left\{\begin{aligned}
		G(D^{2}u,Du,x)&=f &&\text{in } \;\;\Omega,\\
			\mathcal{B}(Du,u,x)&=g &&\text{on } \;\;\partial\Omega,
		\end{aligned}\right.
	\end{equation}
	if the following conditions hold:
	\begin{enumerate}
	\item[(i)] If for all $x_0 \in \Omega$:
	\begin{itemize}	
		\item  [{\rm$\bullet$}] either there exists $\delta>0$ such that $u$ is constant in $B_{\delta}(x_{0})$ and $f(x)\geq 0$ (resp. $f(x)\leq 0$) for all $x\in B_{\delta}(x_{0})$;
		\item  [{\rm$\bullet$}]  or for all $\varphi \in C^2\left(\Omega\right)$ such that $ u -\varphi$ attains a local minimum (resp. local maximum) at $x_0$ and $D\varphi(x_{0})\neq 0$, it holds  
		$$
		G(D^2\varphi(x_0),D\varphi(x_0),x_0) \geq  0 \quad({\rm resp.}\; G(D^2\varphi(x_0),D\varphi(x_0),x_0) \leq  0).
		$$
	\end{itemize}
	\item[(ii)] If for any $x_0\in \partial\Omega$, for all $\varphi\in C^{2}(\overline{\Omega})$ such that $u-\varphi$ has a local maximum (resp., minimum) at $x_0$, then 
	\begin{align*}
		\mathcal{B}(D\varphi(x_0),\varphi(x_0),x_0)\leq  \;g(x_{0}) \quad (\text{resp.,} \;\; \mathcal{B}(D\varphi(x_0),\varphi(x_0),x_0)\geq  \;g(x_{0})).
	\end{align*}	
\end{enumerate}
	Finally, a function $u\in C(\overline{\Omega})$ is said to be a viscosity solution of \eqref{model22} if it is simultaneously a viscosity supersolution and a viscosity subsolution. 
\end{definition}
This definition of viscosity solution was proposed by Birindelli and Demengel in
\cite{Birindelli2004,Birindelli2006,Birindelli2010JDE} for the singular operators. Note that in the case $i(\Phi)\geq 0$, Definition \ref{dingyi} is equivalent to the classical definition of viscosity solution (cf. \cite[Lemma 2.1]{Davil2010}). In addition, the notion of viscosity solutions for oblique derivative problems was introduced first by P.-L. Lions \cite{Lions1985Duke}.

We next give a known result that converts the equation from the singular case to the degenerate case. 
\begin{lemma} \label{lem1} {\rm(cf. \cite[Proposition 2.1]{Huo2026}} 
	Assume \eqref{A1}-\eqref{A3} hold with $-1<i(\Phi)<0$, and $f\in C(\Omega)\cap L^{\infty}(\Omega)$. Let $u$ be a viscosity solution to 
	\begin{align*}
		\Phi(|Du|,x)F(D^{2}u,x)+\mathcal{H}(Du,x)=f \quad \text{in }\;\; \Omega.
	\end{align*}
Then $u$ solves 
\begin{align*}
	\hat{\Phi}(|Du|,x)F(D^{2}u,x)+|Du|^{-i(\Phi)} \mathcal{H}(Du,x)=f|Du|^{-i(\Phi)} 
	\quad \text{in } \;\;\Omega
\end{align*}
in the viscosity sense, where $\hat{\Phi}(t,x):=\frac{\Phi(t,x)}{t^{i(\Phi)}}$ satisfies \eqref{A2} with the parameters 
\begin{align*}
	i(\hat{\Phi})=0, \quad s(\hat{\Phi})=s(\Phi)-i(\Phi)\geq0, \quad \hat{L}=L, \quad \hat{\nu}_{0}={\nu}_{0}, \quad \hat{\nu}_{1}={\nu}_{1}.
\end{align*}
\end{lemma}

With the aid of Lemma \ref{lem1}, we can treat the singular and degenerate regimes within a unified framework. Specifically, we will consider the following modified problem
	\begin{equation}\label{8866}
		\left\{\begin{aligned}
			\hat{\Phi}(\abs{Du}, x)F(D^2 u, x)+|Du|^{\vartheta}\mathcal{H}(Du,x)&=f|Du|^{\vartheta} &&\text{in } \;\;\Omega_{1},\\
			\mathcal{B}(Du,u,x)&=g &&\text{on }\;\; \partial\Omega_{1},
		\end{aligned}\right.
	\end{equation}
	where $\hat{\Phi}$ satisfies \eqref{A3} with the constants $(i(\hat{\Phi}),s(\hat{\Phi}),L,\nu_{0},\nu_{1})$ and  $\vartheta\in[0,1)$. More precisely,  
	\begin{equation}\label{upperbarrier}
		(i(\hat{\Phi}), s(\hat{\Phi}), \vartheta):=
		\begin{cases}
			(i({\Phi}), s({\Phi}), 0) &\text{for}\;\;  i(\Phi)\geq 0, \\
			(0, s({\Phi})-i({\Phi}), -i({\Phi})) &\text{for}\;\; -1<i(\Phi)<0.
		\end{cases}
	\end{equation}
Observe that $\vartheta-i(\hat{\Phi})=-i(\Phi)$.
	\begin{proposition}
		In order to prove Theorem \ref{thm1}, it is enough to prove that $u\in C^{1,\alpha}(\overline{\Omega_{1/2}})$ with
		\begin{equation}\label{holderzhibiao}
			\alpha\in (0,\alpha_{0})\cap(0,\alpha^{\prime})\cap \left(0,\frac{1}{1+s(\hat{\Phi})}\right].
		\end{equation}
\end{proposition}	
\section{$C^{1,\alpha}$ regularity up to the boundary}\label{sec3333}
This section is dedicated to prove our main result concerning the global $C^{1,\alpha}$ regularity.

\subsection{H\"{o}lder regularity up to the boundary}\label{2asection4}
In this subsection, with the aid of Lemma \ref{lem2.1}, we establish the following boundary H\"older regularity. 
\begin{lemma}\label{holder}
	Let $\Omega$ be a bounded $C^{1}$ domain and $0\in \partial \Omega$. Assume \eqref{A1}--\eqref{A3}, $f\in C(\Omega)\cap L^{\infty}(\Omega)$, $g\in C(\partial\Omega)$ and $\beta\in C(\partial\Omega;\mathbb{R}^{n})$. Let $u$ be a viscosity solution to 
	\begin{equation}\label{0001}
		\left\{\begin{aligned}
			\hat{\Phi}\left(\left|Du-q\right|,x\right)F(D^{2}u,x)+|Du-q|^{\vartheta}\mathcal{H}(Du-q,x) &=f|Du-q|^{\vartheta} &&\text{in } \;\;\Omega_{1}, \\
			\beta\cdot Du&=g &&\text{on }\;\; \partial\Omega_{1},
		\end{aligned}\right.
	\end{equation}
with $\|u\|_{L^{\infty}(\Omega_{1})}\leq 1$ and $\|g\|_{L^{\infty}(T_{1})}\leq 1$. Assume that 
\begin{equation}\label{1719}
\max\biggl\{(\mathcal{K}+\|f\|_{L^{\infty}(\Omega_{1})})\left(1+|q|^{(\vartheta-i(\hat{\Phi}))_{+}}\right),\mathcal{M}\left(1+|q|^{(m-i({\Phi}))_{+}}\right)\biggr\}\leq \mathcal{A}_{0}
\end{equation}
for a universal constant $\mathcal{A}_{0}\leq \frac{2\Lambda \nu_{0}}{L}$. Then there exists $\iota=\iota(\mu_{0})\in(0, \mu_{0}/2]$ such that, if $[\psi]_{C^{1}}\leq \iota$, then $u\in C^{{\alpha}^{\prime}}(\overline{\Omega_{1/2}})$ for some ${\alpha}^{\prime}={\alpha}^{\prime}(d, \lambda, \Lambda,\mu_{0})\in(0,1)$ with the estimate
	\begin{align*}
		\norm{u}_{C^{{\alpha}^{\prime}}(\overline{\Omega_{1/2}})}\leq C,
	\end{align*}
	where $C$ is a universal positive constant independent of $q$.
\end{lemma}
\begin{proof} In order to apply Lemma \ref{lem2.1}, it is enough to show that $u$ satisfies
	\begin{equation}\label{178}
		\mathcal{L}^{+}_{\lambda, \Lambda}(D^{2}u,Du)\geq -C_{0}\quad  \text{in }\;\; \left\{|Du-q|>1\right\}\cap\Omega_{1}
	\end{equation}
and
	\begin{equation}\label{179}
	\mathcal{L}^{-}_{\lambda, \Lambda}(D^{2}u,Du)\leq C_{0}\quad  \text{in }\;\; \left\{|Du-q|>1\right\}\cap\Omega_{1}
\end{equation}
in the viscosity sense for some positive constant $C_{0}$ to be determined later. We first prove \eqref{178}.
Let $\varphi\in C^{2}$ be a test function such that $u-\varphi$ attains a local maximum at $x_{0}\in \Omega_{1}$ and $|D\varphi(x_{0})-{q}|>1$. By the definition of viscosity subsolution, we have
\begin{equation*}
	\hat{\Phi}\left(\left|D\varphi(x_{0})-{q}\right|,x_{0}\right)F(D^{2}\varphi(x_{0}),x_{0})+|D\varphi(x_{0})-{q}|^{\vartheta}\mathcal{H}(D\varphi(x_{0})-{q},x_{0}) \geq f(x_0)|D\varphi(x_{0})-{q}|^{\vartheta}.
\end{equation*}
Then it follows from \eqref{A1} and $F(0,x_{0})=0$ that
\begin{equation}\label{221}
	\begin{split}
		\mathcal{L}^{+}_{\lambda, \Lambda}(D^{2}\varphi(x_{0}),D\varphi(x_{0}))&= 	\mathcal{P}^{+}_{\lambda, \Lambda}(D^{2}\varphi(x_{0}))+\Lambda|D\varphi(x_{0})|\\
		&\geq F(D^{2}\varphi(x_{0}),x_{0})-F(0,x_{0})+\Lambda|D\varphi(x_{0})|\\
		&\geq \frac{|D\varphi(x_{0})-{q}|^{\vartheta}\left(f(x_{0})-\mathcal{H}(D\varphi(x_{0})-{q},x_{0}) \right)}{\hat{\Phi}\left(\left|D\varphi(x_{0})-{q}\right|,x_{0}\right)}+\Lambda|D\varphi(x_{0})|\\
		&=:\mathcal{J}_{1}+\Lambda|D\varphi(x_{0})|.
	\end{split}
\end{equation}
Since $|D\varphi(x_{0})-{q}|>1$, applying \eqref{A2} and \eqref{A3} to deduce
\begin{equation}\label{222}
	\begin{split}
		|\mathcal{J}_{1}|&\leq \frac{L}{\nu_{0}}\left(\mathcal{K}+\|f\|_{L^{\infty}(\Omega_{1})}\right)|D\varphi(x_{0})-{q}|^{\vartheta-i(\hat{\Phi})}+\frac{L}{\nu_{0}}\mathcal{M}|D\varphi(x_{0})-{q}|^{m+\vartheta-i(\hat{\Phi})}\\
		&=\frac{L}{\nu_{0}}\left(\mathcal{K}+\|f\|_{L^{\infty}(\Omega_{1})}\right)|D\varphi(x_{0})-{q}|^{-i({\Phi})}+\frac{L}{\nu_{0}}\mathcal{M}|D\varphi(x_{0})-{q}|^{m-i({\Phi})}\\
		&=:\mathcal{J}_{11}+\mathcal{J}_{12},
	\end{split}
\end{equation}
where the fact $\vartheta-i(\hat{\Phi})=-i(\Phi)$ is used.\\
{\bf Estimate of $\mathcal{J}_{11}$.} If $i(\Phi)\geq 0$, then applying $|D\varphi(x_{0})-{q}|>1$ and \eqref{1719} to obtain
\begin{equation*}
	\begin{split}
		\mathcal{J}_{11}\leq \frac{L}{\nu_{0}}\left(\mathcal{K}+\|f\|_{L^{\infty}(\Omega_{1})}\right)\leq \frac{L}{\nu_{0}}\mathcal{A}_{0}.
	\end{split}
\end{equation*}
If $-1<i(\Phi)< 0$, for the case $|D\varphi(x_{0})|<1$, then it follows from the basic inequality 
\begin{equation}\label{basicinequality}
(a_{1}+a_{2})^{\kappa}\leq a_{1}^{\kappa}+a_{2}^{\kappa}\quad{\rm for \;any\;} a_{1},a_{2}>0,0<\kappa\leq 1
\end{equation}
and \eqref{1719} that
\begin{equation*}
	\begin{split}
		\mathcal{J}_{11}&\leq \frac{L}{\nu_{0}}\left(\mathcal{K}+\|f\|_{L^{\infty}(\Omega_{1})}\right)\left(|D\varphi(x_{0})|^{-i(\Phi)}+|q|^{-i(\Phi)}\right)\\
		&\leq \frac{L}{\nu_{0}}\left(\mathcal{K}+\|f\|_{L^{\infty}(\Omega_{1})}\right)\left(1+|q|^{-i(\Phi)}\right)\\
		&\leq \frac{L}{\nu_{0}}\mathcal{A}_{0}.
	\end{split}
\end{equation*}
For the case $|D\varphi(x_{0})|\geq 1$, then it follows from \eqref{1719} that
\begin{equation*}
	\begin{split}
		\mathcal{J}_{11}&\leq \frac{L}{\nu_{0}}\mathcal{A}_{0}+\frac{L}{\nu_{0}}\left(\mathcal{K}+\|f\|_{L^{\infty}(\Omega_{1})}\right)|D\varphi(x_{0})|^{-i(\Phi)}\\
		&\leq 	\frac{L}{\nu_{0}}\mathcal{A}_{0}+\frac{L}{\nu_{0}}\left(\mathcal{K}+\|f\|_{L^{\infty}(\Omega_{1})}\right)|D\varphi(x_{0})|.
	\end{split}
\end{equation*}
In summary, we deduce
\begin{equation*}
	\begin{split}
		\mathcal{J}_{11}&\leq 
		\begin{cases}
			\frac{L}{\nu_{0}}\mathcal{A}_{0}+\frac{L}{\nu_{0}}\left(\mathcal{K}+\|f\|_{L^{\infty}(\Omega_{1})}\right)|D\varphi(x_{0})| &\text{for}\;\;|D\varphi(x_{0})|\geq 1,\\
			\frac{L}{\nu_{0}}\mathcal{A}_{0} &\text{for}\;\;|D\varphi(x_{0})|<1.
		\end{cases}
	\end{split}
\end{equation*}
{\bf Estimate of $\mathcal{J}_{12}$.} If $0<m\leq i(\Phi)$, then it follows from $|D\varphi(x_{0})-{q}|>1$ and \eqref{1719} that 
\begin{equation*}
	\begin{split}
		\mathcal{J}_{12}\leq \frac{L}{\nu_{0}}\mathcal{M}.
	\end{split}
\end{equation*}
If $i(\Phi)<m\leq 1+i(\Phi)$, for the case $|D\varphi(x_{0})|<1$, a combination of \eqref{basicinequality} and \eqref{1719} yields 
\begin{equation*}
	\begin{split}
		\mathcal{J}_{12}\leq \frac{L}{\nu_{0}}\mathcal{M}\left(|D\varphi(x_{0})|^{m-i(\Phi)}+|q|^{m-i(\Phi)}\right)
		\leq \frac{L}{\nu_{0}}\mathcal{M}\left(1+|q|^{m-i(\Phi)}\right)\leq \frac{L}{\nu_{0}}\mathcal{A}_{0}.
	\end{split}
\end{equation*}
For the case $|D\varphi(x_{0})|\geq 1$, it is easy to deduce that
\begin{equation*}
	\begin{split}
		\mathcal{J}_{12}\leq \frac{L}{\nu_{0}}\mathcal{A}_{0}+\frac{L}{\nu_{0}}\mathcal{M}|D\varphi(x_{0})|^{m-i(\Phi)}\leq \frac{L}{\nu_{0}}\mathcal{A}_{0}+\frac{L}{\nu_{0}}\mathcal{M}|D\varphi(x_{0})|.
	\end{split}
\end{equation*}
In summary, we have
\begin{equation*}
	\begin{split}
		\mathcal{J}_{12}&\leq 
		\begin{cases}
			\frac{L}{\nu_{0}}\mathcal{A}_{0}+\frac{L}{\nu_{0}}\mathcal{M}|D\varphi(x_{0})| &\text{for}\;\;|D\varphi(x_{0})|\geq 1,\\
			\frac{L}{\nu_{0}}\mathcal{A}_{0} &\text{for}\;\;|D\varphi(x_{0})|<1.
		\end{cases}
	\end{split}
\end{equation*}

Consequently, for the case $|D\varphi(x_{0})|<1$, substituting the estimates of $J_{11}$ and $J_{12}$ into \eqref{221}, we derive
\begin{equation*}
	\begin{split}
		\mathcal{L}^{+}_{\lambda, \Lambda}(D^{2}\varphi(x_{0}),D\varphi(x_{0}))&\geq -\mathcal{J}_{11}-\mathcal{J}_{12}+\Lambda|D\varphi(x_{0})|\\
		&\geq -\frac{2L}{\nu_{0}}\mathcal{A}_{0}.
	\end{split}
\end{equation*}
On the other hand, for the case $|D\varphi(x_{0})|\geq 1$, we can get
\begin{equation*}
	\begin{split}
		\mathcal{L}^{+}_{\lambda, \Lambda}(D^{2}\varphi(x_{0}),D\varphi(x_{0}))
		&\geq -\frac{2L}{\nu_{0}}\mathcal{A}_{0}+\left(\Lambda-\frac{L}{\nu_{0}}\left(\mathcal{K}+\|f\|_{L^{\infty}(\Omega_{1})}+\mathcal{M}\right)\right)|D\varphi(x_{0})| \\
		&\geq -\frac{2L}{\nu_{0}}\mathcal{A}_{0}.
	\end{split}
\end{equation*}

Combining both cases, we take $C_{0}=\max\left\{4\Lambda,1\right\}$ and conclude that
\begin{align*}
	\mathcal{L}_{\lambda, \Lambda}^{+}(D^{2}u, Du) \geq -C_{0} \quad \text{in } \{|Du-q|>1\}\cap\Omega_{1}.
\end{align*}
By an analogous argument, we can also show that \eqref{179} holds.
Therefore, the desired conclusion immediately follows from Lemma \ref{lem2.1}. 
\end{proof}
\subsection{Approximation lemma}
This subsection is devoted to proving a key approximation lemma, which plays a paramount role in our forthcoming geometric argument. 
\begin{lemma}\label{lem01}
	 Assume \eqref{A1}-\eqref{A5} hold and $u$ is a viscosity solution to \eqref{0001} with $\|u\|_{L^{\infty}(\Omega_{1})}\leq 1$. Given $\varepsilon>0$, there exists $\sigma=\sigma(d, \lambda, \Lambda, i(\Phi), L, \nu_{0}, \mu_{0},\varepsilon)>0$ such that if
	\begin{equation*}
		\begin{split}
			\max\Bigg\{
			&\left(\fint_{\Omega_{1}}\Abs{{\rm osc}_{{F}}(x)}^{d}\text{d}x\right)^{1/d}, (1+|q|^{(\vartheta-i(\hat{\Phi}))_+})\left(\norm{f}_{L^{\infty}(\Omega_{1})}+\mathcal{K}\right),\\ &\mathcal{M}\left(1+|q|^{(m-i({\Phi}))_{+}}\right), \norm{g}_{L^{\infty}(T_{1})}, [\beta-\beta_{0}]_{C^{\alpha^{\prime}}(T_{1})}, \norm{\psi}_{C^{1}(T_{1})}
			\Bigg\} \leq \sigma,
		\end{split}
	\end{equation*}
	then there exists a function $v\in C^{1, \alpha_{0}}\left(\overline{B_{3/4}^{+}}\right)$  satisfying
	\begin{equation} \label{111tiaohe}
		\left\{
		\begin{array}{rcrcl}
			{F}(D^2 v,0)= 0 & \text{in} & B_{3/4}^{+}, \\
			\beta_{0}\cdot Dv= 0 & \text{on} & T_{3/4},
		\end{array}
		\right.
	\end{equation}
	where $\beta_{0}:=\beta(0)$, such that
	\begin{align*}
		\norm{u-v}_{L^{\infty}\left(\Omega_{1/2}\right)}\leq\varepsilon.
	 \end{align*}
\end{lemma}
\begin{proof}
	The proof is based on a contradiction argument. Suppose the thesis of the lemma fails to hold. Then there exist $\varepsilon_{0}>0$ and sequences $\{F_{j}\}_{j\in \mathbb{N}}$, $\{\hat{\Phi}_{j}\}_{j\in \mathbb{N}}$, $\{f_{j}\}_{j\in \mathbb{N}}$, $\{u_{j}\}_{j\in \mathbb{N}}$, $\{\mathcal{H}_{j}\}_{j\in \mathbb{N}}$, $\{g_{j}\}_{j\in \mathbb{N}}$, $\{q_{j}\}_{j\in \mathbb{N}}$, $\{\beta_{j}\}_{j\in \mathbb{N}}$, $\{\Omega_{j}\}_{j\in \mathbb{N}}$ such that
	\begin{itemize}
		\item   [{\rm$({{\rm i}})$}]  $F_{j}:S^{d}\times \left(\Omega_{j}\right)_{1} \rightarrow \mathbb{R}$ is uniformly $(\lambda,\Lambda)$-elliptic;
		\item [{\rm$({{\rm ii}})$}] $\hat{\Phi}_{j}\in C\left([0,\infty)\times (\Omega_{j})_{1}\right)$ such that the map $t\mapsto \frac{\hat{\Phi}_{j}(t,y)}{t^{i(\hat{\Phi})}}$ is almost non-decreasing and the map $t\mapsto \frac{\hat{\Phi}_{j}(t,y)}{t^{s(\hat{\Phi})}}$ is almost non-increasing  with the same constant $L\geq 1$ in $(0,\infty)$;	
		\item [ {\rm$({{\rm iii}})$}] $\mathcal{H}_{j}:\rn \times (\Omega_{j})_{1}$ is continuous and there exist constants $\mathcal{K}_{j},\mathcal{M}_{j}>0$ such
	that
	\begin{equation*}
		|\mathcal{H}_{j}(t,x)|\leq \mathcal{K}_{j}+\mathcal{M}_{j}|t|^{m} \quad  {\rm for\; every\;}(t,x) \in\rn\times \left(\Omega_{j}\right)_{1};
	\end{equation*}	
	\item [ {\rm$({{\rm iv}})$}] $u_{j}$ is a viscosity solution of 
	\begin{equation*}
		\left\{\begin{aligned}
			\hat{\Phi}_{j}\left(\left|Du_{j}-q_{j}\right|,x\right)F_{j}(D^{2}u_{j},x)+|Du_{j}-q_{j}|^{\vartheta}\mathcal{H}_{j}(Du_{j}-q_{j},x) &=f_{j}|Du_{j}-q_{j}|^{\vartheta} &&\text{in } \;\;(\Omega_{j})_{1}, \\
			\beta_{j}\cdot Du_{j}&=g_{j} &&\text{on }\;\; \partial(\Omega_{j})_{1},
		\end{aligned}\right.
	\end{equation*}
	with $ \|u_{j}\|_{L^{\infty}((\Omega_{j})_{1})}\leq 1$, where $\Omega_{j}$ is a $C^{1}$ domain;
	\item [ {\rm$({{\rm v}})$}] and
	\begin{equation*}
	\begin{split}
		\max\Bigg\{
		&\left(\fint_{(\Omega_{j})_{1}}\Abs{{\rm osc}_{{F_{j}}}(x)}^{d}\text{d}x\right)^{1/d}, (1+|q_{j}|^{(\vartheta-i(\hat{\Phi}))_+})\left(\norm{f_{j}}_{L^{\infty}((\Omega_{j})_{1})}+\mathcal{K}_{j}\right),\\ &\mathcal{M}_{j}\left(1+|q_{j}|^{(m-i({\Phi}))_{+}}\right), \norm{g_{j}}_{L^{\infty}(T_{1})}, [\beta_{j}-\beta_{0}]_{C^{\alpha^{\prime}}(T_{1})}, \norm{\psi_{j}}_{C^{1}(T_{1})}
		\Bigg\} \leq\frac{1}{j}.
	\end{split}
\end{equation*}
\end{itemize}	
However, it holds 
\begin{equation}\label{22daomaodun}
\|u_{j}-v_{j}\|_{L^{\infty}\left((\Omega_{j})_{1/2}\right)}>\varepsilon_{0}
\end{equation}
for any $v_{j}$ satisfying the corresponding conditions.

By the uniform ellipticity assumption on the operator $F_{j}$ and (v), we may assume (after passing to a subsequence if necessary)that $F_{j}(M,x)\rightarrow F_{\infty}(M,0)$ locally uniformly in $S^{d}\times B_{1}^{+}$, where $F_{\infty}$ is a uniformly $(\lambda,\Lambda)$-elliptic operator. Moreover, by the boundary H\"older estimates in Lemma \ref{holder} and the Arzel\`a--Ascoli theorem, there exists a function $u_{\infty}\in C(\overline{B_{1}^{+}})$ such that $u_{j}$ converges to $u_{\infty}$ locally uniformly up to a subsequence.
Additionally, $\beta_{j}\to \beta_\infty$ for some constant vector $\beta_\infty \in \mathbb{R}^d$ and $(\Omega_{j})_{1}\to B_{1}^{+}$.
Now, by arguing as \cite[Lemma 4.1]{Huo2026}, we conclude that $u_{\infty}$ is a viscosity solution to
\begin{equation*}
	\left\{\begin{aligned}
		F_{\infty}(D^{2}u_{\infty},0) &=0 &&\text{in } B_{1}^{+}, \\
		\beta_\infty\cdot Du_{\infty}&=0 &&\text{on } T_{1}.
	\end{aligned}\right.
\end{equation*}
To proceed, let $v_{j}$ is a viscosity solution of
$$
\left\{
\begin{array}{rclcl}
	F_{k}(D^2 v_{j},0) &=& 0 & \mbox{in} & {B}^+_{\frac{3}{4}}, \\
	\beta_{j}(0) \cdot D v_{j} &=& 0 & \mbox{on} & {T}_{\frac{3}{4}},\\
	v_j &=& u_j & \mbox{on} & \partial {B}^{+}_{\frac{3}{4}}\setminus {T}_{\frac{3}{4}}.
\end{array}
\right.
$$
Existence and uniqueness follow from \cite[Theorem 3.3]{LiZhang}. By stability, $v_j\to v_\infty$ locally uniformly, where $v_\infty$ satisfies
$$
\left\{
\begin{array}{rclcl}
	F_{\infty}(D^2 v_{\infty},0) &=& 0 & \mbox{in} & {B}^+_{\frac{3}{4}}, \\
	\beta_0 \cdot D v_{\infty} &=& 0 &\mbox{on}& {T}_{\frac{3}{4}} \\
	v_{\infty} &=& u_{\infty} & \mbox{on} & \partial {B}^{+}_{\frac{3}{4}}\setminus {T}_{\frac{3}{4}}.
\end{array}
\right.
$$
Since the solution of the above equation is unique, we get $u_{\infty}=v_{\infty}$, which contradicts with \eqref{22daomaodun} for $j$ sufficiently large. This completes the proof of the desired result.
\end{proof}

\begin{lemma}\label{lem11}
 Assume \eqref{A1}-\eqref{A5} hold. Let $u$ be a viscosity solution to \eqref{0001} with $\|u\|_{L^{\infty}(\Omega_{1})}\leq 1$. For every ${\alpha}\in(0, \alpha_{0})$, there exists constant $\sigma>0$ such that if 
\begin{equation}\label{080}
	\begin{split}
		\max\Bigg\{
		&\left(\fint_{\Omega_{1}}\Abs{{\rm osc}_{{F}}(x)}^{d}{\rm d}x\right)^{1/d}, (1+|q|^{(\vartheta-i(\hat{\Phi}))_+})\left(\norm{f}_{L^{\infty}(\Omega_{1})}+\mathcal{K}\right),\\ &\mathcal{M}\left(1+|q|^{(m-i({\Phi}))_{+}}\right), \norm{g}_{L^{\infty}(T_{1})}, [\beta-\beta_{0}]_{C^{\alpha^{\prime}}(T_{1})}, \norm{\psi}_{C^{1}(T_{1})}
		\Bigg\} \leq \sigma,
	\end{split}
\end{equation}
then there exist a universal constant $0<\rho<\frac{1}{2}$ and an affine function $\ell_{1}(x)=u(0)+\mathfrak{b}_{1}\cdot x$ such that
\begin{align}\label{011}
	\norm{u-\ell_{1}}_{L^{\infty}(\Omega_{\rho})}\leq \rho^{1+{\alpha}}, \quad  \beta_{0}\cdot \mathfrak{b}_{1}=0, \quad  |\mathfrak{b}_{1}|\leq {C_{e}},    
\end{align}
where constant ${C_e}=C_{e}(d,\lambda,\Lambda,\mu_{0})>0$.
\end{lemma}
\begin{proof}
	Let $\varepsilon>0$ be a fixed number to be chosen later and  $\sigma$ be the corresponding constant in Lemma \ref{lem01}. From Lemma \ref{lem01}, there exists a function $v\in C^{1,\alpha_{0}}(B_{3/4}^{+})$ satisfying \eqref{111tiaohe} such that  
	\begin{equation}\label{2bijinxiaoliang}
		\|u-v\|_{L^{\infty}(\Omega_{1/2})}\leq \varepsilon.
	\end{equation} 
	By virtue of the $C^{1,\alpha_{0}}$ estimate of $v$ at $0$, there exists a universal positive constant $C_{e}$ such that
	\begin{equation}\label{zhengzexing}
	\Abs{v(x)-l(x)}\leq C_{e}|x|^{1+\alpha_{0}},
	\end{equation}
	\begin{equation}\label{180}
		\beta_{0}\cdot Dv(0)=0\quad |Dv(0)|\leq C_{e}.
	\end{equation}
where $l(x):=v(0)+Dv(0)\cdot x$. Next, let us denote $\mathfrak{b}_{1}:=Dv(0)$ and $\ell(x):=u(0)+\mathfrak{b}_{1}\cdot x$. A combination of the triangle inequality with \eqref{2bijinxiaoliang} and \eqref{zhengzexing} yields that
	\begin{equation}\label{5.1}
		\begin{split}
			\|u-\ell\|_{{L^{\infty}(\Omega_{r})}}&\leq \|u-v\|_{{L^{\infty}(\Omega_{r})}}+\|v-l\|_{{L^{\infty}(\Omega_{r})}}+|u(0)-v(0)|\\
			&\leq 2\varepsilon+C_{e}\rho^{1+\alpha_{0}}\\
			&\leq \rho^{1+\alpha}.
		\end{split}
	\end{equation}
as long as we make the following universal choices
\begin{equation*}
	\rho\in \left(0,\min\left\{\frac{1}{2},\left(\frac{1}{3C_{e}}\right)^{{\alpha}_{0}-\alpha}\right\}\right)\quad {\rm and}\quad \varepsilon\in \left(0,\frac{1}{3}\rho^{1+\alpha}\right),
\end{equation*}	
	Consequently, combining \eqref{180} with \eqref{5.1}, we obtain the desired result \eqref{011}. This completes the proof of the desired result.
\end{proof}

\subsection{Pointwise boundary $C^{1,\alpha}$ regularity}
With the help of Lemma \ref{lem11}, we establish the following pointwise boundary $C^{1,\alpha}$ regularity estimate. 
\begin{proposition}\label{prop3.4}
	 Assume \eqref{A1}--\eqref{A5} hold with $\nu_{0}=\nu_{1}=1$ and $u$ is a viscosity solution to \eqref{0001} with $\|u\|_{L^{\infty}(\Omega_{1})}\leq 1$ and $|q|\leq \frac{C_{e}}{1-\rho^{\alpha}}$. Given $\alpha$ as in \eqref{holderzhibiao},
	there exist a universal constant $\delta>0$ such that if 
\begin{equation}\label{071}
	\begin{split}
		\max\Bigg\{
		&\left(\fint_{\Omega_{1}}\Abs{{\rm osc}_{{F}}(x)}^{d}{\rm d}x\right)^{1/d}, \norm{f}_{L^{\infty}(\Omega_{1})}+\mathcal{K},\\
		&\mathcal{M}, \norm{g}_{C^{\alpha^{\prime}}(T_{1})}, [\beta-\beta_{0}]_{C^{\alpha^{\prime}}(T_{1})}, \norm{\psi}_{C^{1}(T_{1})}
		\Bigg\} \leq \delta,
	\end{split}
\end{equation}
	then $u\in C^{1,\alpha}(0)$ with the estimate
	\begin{equation*}
		\sup\limits_{x\in \Omega_{r}}\Abs{u(x)-\left(u(0)+Du(0)\cdot x\right)}\leq Cr^{1+\alpha}\quad {\rm for \;all\;} r\in (0,\rho],
	\end{equation*}
	where constant $\rho\in(0,\frac{1}{2})$ coming from Lemma \ref{lem11} and $C$ is a universal positive constant.
\end{proposition}
\begin{proof}
	To prove that $u$ is $C^{1,\alpha}$ at 0, we only need to prove that there exist constants $\rho\in(0,\frac{1}{2})$ and $C_{e}>0$, and a sequence of affine functions 
	$$\ell_{j}(x)=\mathfrak{a}_{j}+\mathfrak{b}_{j}\cdot x$$
	with $\left\{\mathfrak{a}_{j}\right\}_{j\in \mathbb{N}}\subset \mathbb{R}$ and $\left\{\mathfrak{b}_{j}\right\}_{j\in \mathbb{N}}\subset \rn$, such that for all $j\in\mathbb{N}$, the following estimates hold:
	\begin{equation}\label{711}
		\|u-\ell_{j}\|_{L^{\infty}(\Omega_{\rho^{j}})}\leq \rho^{j(1+\alpha)},
	\end{equation}
	\begin{equation}\label{712}
	\mathfrak{a}_{j}=u(0),\quad \beta_{0}\cdot \mathfrak{b}_{j}=0,\quad |\mathfrak{b}_{j}-\mathfrak{b}_{j-1}|\leq C_{e}\rho^{(j-1)\alpha}.
	\end{equation}
	As usual, we prove this by means of an induction argument.
	For clarity of presentation, we divide the proof into three steps.\\
	{\bf Step 1. Basis of induction.} Initially, we choose $\delta>0$ small enough such that
	\begin{equation}\label{0801}
	\delta<\frac{2C_{e}}{1-\rho^{\alpha}}, \quad \delta\left(1+\frac{C_{e}}{1-\rho^{\alpha}}\right)\leq \sigma,
	\end{equation}
	\begin{equation}\label{0808}
		L \delta\left(1+\max\left\{\left(\frac{2C_{e}}{1-\rho^{\alpha}}\right)^{(\vartheta-i(\hat{\Phi}))_+},\left(\frac{2C_{e}}{1-\rho^{\alpha}}\right)^{(m-i({\Phi}))_+}\right\}\right)\leq \sigma
	\end{equation}
	with $\sigma$ coming from Lemma \ref{lem11}.
	With such choice, it follows from \eqref{071} and $|q|\leq \frac{C_{e}}{1-\rho^{\alpha}}$ that the assumption \eqref{080} holds. Hence, the conclusion \eqref{011} holds. Let $\mathfrak{b}_{0}=0$, the case $j=1$ follows at once from Lemma \ref{lem11}.\\
	{\bf Step 2. Induction process.} Suppose that the conclusion holds for $j=1,2,...,k$; we examine the claim also holds for $j=k+1$. To this end, we introduce an auxiliary function $u_{k}$ as 
	\begin{equation*}
		u_{k}(x):=\frac{u\left(\rho^{k}x\right)-\ell_{k}\left(\rho^{k}x\right)}{\rho^{k(1+\alpha)}}.
	\end{equation*}
	Then $u_{k}$ is a viscosity of 
	\begin{equation*}
		\left\{\begin{aligned}
			\hat{\Phi}_{k}(\Abs{D{u_{k}}-q_{k}},x){F}_{k}(D^2 {u_{k}}, x)+\Abs{D{u_{k}}-q_{k}}^{\vartheta}\mathcal{H}_{k}({D{u_{k}}-q_{k}},x)&= {f_{k}}\Abs{D{u_{k}}-q_{k}}^{\vartheta} && \text{in} \;\; \left(\frac{1}{\rho^{k}}\Omega\right)\cap B_{1}, \\
			\beta_{k}\cdot Du_{k} &=  g_{k} && \text{on} \; \partial\left(\frac{1}{\rho^{k}}\Omega\right)\cap B_{1},
		\end{aligned}\right.
	\end{equation*}
	where 
	\begin{align*}
		{F_{k}}(X,x):=&\rho^{k(1-\alpha)}F\left(\rho^{k(\alpha-1)}X,\rho^{k}x\right),\quad
		{\hat{\Phi}_{k}}(t,x):=\frac{\hat{\Phi}\left(\rho^{k\alpha}t,\rho^{k}x\right)}{\hat{\Phi}\left(\rho^{k\alpha},\rho^{k}x\right)},\\
		{\mathcal{H}_{k}}(t,y):=&\frac{\rho^{k(1-\alpha(1-\vartheta)))}}{\hat{\Phi}\left(\rho^{k\alpha},\rho^{k}x\right)}\mathcal{H}(\rho^{k\alpha}t,\rho^{k}x),\quad
		{f_{k}}(y):=\frac{\rho^{k(1-\alpha(1-\vartheta))}}{\hat{\Phi}\left(\rho^{k\alpha},\rho^{k}x\right)}f(\rho^{k}x),\\
		{\beta_{k}}(x):=&\beta(\rho^{k}x),\quad
		{g_{k}}(x):=\frac{g\left(\rho^{k}x\right)-\beta\left(\rho^{k}x\right)\cdot \mathfrak{b}_{k}}{\rho^{k\alpha}},\\
		\psi_{k}(x):=&\psi_{\frac{1}{\rho^{k}}\Omega}(x)=\frac{\psi(\rho^{k}x)}{\rho^{k}},\quad
		q_{k}:=\frac{q-\mathfrak{b}_{k}}{\rho^{k\alpha}}.
	\end{align*}
	It is obvious that ${F_{k}}$ is a uniformly $(\lambda,\Lambda)$-elliptic operator and 
	\begin{equation*}
		\begin{split}
			{\rm osc}_{{F_{k}}}(x)=\sup\limits_{M\in S^{d}\setminus \{0\}}\frac{\Abs{F\left(\rho^{k(\alpha-1)}M,\rho^{k}x\right)-F\left(\rho^{k(\alpha-1)}M,0\right)}}{\rho^{k(\alpha-1)}\|M\|}={\rm osc}_{{F}}(\rho^{k}x),
		\end{split}
	\end{equation*}
	which immediately yields that
	\begin{equation*}
		\begin{split}
			\left(\fint_{\left(\frac{1}{\rho^{k}}\Omega\right)\cap B_{1}}\Abs{{\rm osc}_{{F_{k}}}(x)}^{d}\text{d}x\right)^{1/d}=\left(\fint_{\Omega_{1}}\Abs{{\rm osc}_{{F}}(x)}^{d}\text{d}x\right)^{1/d}\leq \delta\leq \sigma.
		\end{split}
	\end{equation*} 
	By induction assumption \eqref{711}, we have
	$$\|{u_{k}}\|_{L^{\infty}\left(\left(\frac{1}{\rho^{k}}\Omega\right)\cap B_{1}\right)}\leq 1.$$
	It is clear that 
	$\|\psi_{k}\|_{C^{1}(T_{\rho^{-k}})}\leq \|\psi\|_{C^{1}(T_{1})} \leq \delta\leq \sigma$ and 
	$$[\beta_{k}-\beta_{0}]_{C^{\alpha^{\prime}}(T_{\rho^{-k}})}\leq \rho^{k\alpha^{\prime}}[\beta-\beta_{0}]_{C^{\alpha^{\prime}}(T_{1})}\leq \delta\leq \sigma.$$ 
	Notice that $\hat{\Phi}_{k}$ still satisfies \eqref{A2} with the same constants $(i(\hat{\Phi}),s(\hat{\Phi}))$ and ${\hat{\Phi}_{k}}(1,\cdot)=1$.
	Applying the properties of $\hat{\Phi}_{k}$, in combination  with \eqref{A3} and $\rho\in(0,\frac{1}{2})$, we deduce that
\begin{equation}\label{081}
	\|{f_{k}}\|_{L^{\infty}\left(\left(\frac{1}{\rho^{k}}\Omega\right)\cap B_{1}\right)}\leq \frac{L\rho^{k(1-\alpha(1-\vartheta))}\|{f}\|_{L^{\infty}\left(\Omega_{1}\right)}}{\rho^{k\alpha s(\hat{\Phi})}},
\end{equation}
	\begin{equation}\label{082}
		|\mathcal{H}_{k}(t,x)|\leq \frac{L\rho^{k(1-\alpha(1-\vartheta))}}{\rho^{k\alpha s(\hat{\Phi})}}\left(\mathcal{K}+\mathcal{M}\rho^{k\alpha m}|t|^{m}\right)=:\mathcal{K}_{k}+\mathcal{M}_{k}|t|^{m}.
	\end{equation} 
	It follows from the induction hypothesis \eqref{712} that		
	\begin{align*}
		|\mathfrak{b}_{k}|\leq |\mathfrak{b}_{0}|+\sum_{j=1}^{k}\abs{\mathfrak{b}_{j}-\mathfrak{b}_{j-1}}\leq C_{e}\sum_{j=1}^{k}\rho^{\alpha(j-1)}\leq \frac{C_{e}}{1-\rho^{\alpha}}.
	\end{align*}
	This together with $|q|\leq \frac{C_{e}}{1-\rho^{\alpha}}$ yields that
	\begin{equation}\label{083}
		|q_{k}|\leq \frac{|q|+|\mathfrak{b}_{k}|}{\rho^{k\alpha}}\leq \frac{2C_{e}}{\rho^{k\alpha}(1-\rho^{\alpha})}.
	\end{equation} 
Utilizing \eqref{071} and \eqref{081}-\eqref{083}, in combination with ${\alpha}\in \left(0, \frac{1}{1+s(\hat{\Phi})}\right]$ and \eqref{0808}, we arrive at
	\begin{equation*}
		\begin{split}
			&(1+|q_{k}|^{(\vartheta-i(\hat{\Phi}))_+})\left(\norm{f_{k}}_{L^{\infty}(\Omega_{1})}+\mathcal{K}_{k}\right)\\\leq& L\delta \rho^{k\left(1-\alpha(1+s(\hat{\Phi})-\vartheta)\right)}\left(1+\rho^{-k\alpha(\vartheta-i(\hat{\Phi}))_{+}}\left(\frac{2C_{e}}{1-\rho^{\alpha}}\right)^{(\vartheta-i(\hat{\Phi}))_+}\right)\\
			\leq& L\delta \rho^{k\left(1-\alpha(1+s(\hat{\Phi})-\vartheta+(\vartheta-i(\hat{\Phi}))_+)\right)}\left(1+\left(\frac{2C_{e}}{1-\rho^{\alpha}}\right)^{(\vartheta-i(\hat{\Phi}))_+}\right)\\
			=&L\delta \rho^{k\left(1-\alpha(1+s(\hat{\Phi})-\min\{\vartheta,i(\hat{\Phi})\})\right)}\left(1+\left(\frac{2C_{e}}{1-\rho^{\alpha}}\right)^{(\vartheta-i(\hat{\Phi}))_+}\right)\\
			=&L\delta \rho^{k\left(1-\alpha(1+s(\hat{\Phi}))\right)}\left(1+\left(\frac{2C_{e}}{1-\rho^{\alpha}}\right)^{(\vartheta-i(\hat{\Phi}))_+}\right)\\
			\leq& L\delta \left(1+\left(\frac{2C_{e}}{1-\rho^{\alpha}}\right)^{(\vartheta-i(\hat{\Phi}))_+}\right)\leq \sigma
		\end{split}
	\end{equation*} 
	\begin{equation*}
		\begin{split}
			&(1+|q_{k}|^{(\vartheta-i(\hat{\Phi}))_+})\mathcal{M}_{k}\\
			\leq& L\delta \rho^{k\left(1-\alpha(1+s(\hat{\Phi})-\vartheta-m+(m-i({\Phi}))_+)\right)}\left(1+\left(\frac{2C_{e}}{1-\rho^{\alpha}}\right)^{(m-i({\Phi}))_+}\right)\\
			=&	
			\begin{cases}
				L\delta \rho^{k\left(1-\alpha(1+s(\hat{\Phi})-m)\right)}\left(1+\left(\frac{2C_{e}}{1-\rho^{\alpha}}\right)^{(m-i({\Phi}))_+}\right) &\text{for}\;\;  0<m\leq i(\Phi)\\
					L\delta \rho^{k\left(1-\alpha(1+s(\hat{\Phi})-\max\left\{0,i(\Phi)\right\})\right)}\left(1+\left(\frac{2C_{e}}{1-\rho^{\alpha}}\right)^{(m-i({\Phi}))_+}\right)  &\text{for}\;\; i(\Phi)<m\leq 1+i(\Phi)
			\end{cases}\\
			\leq &L\delta \left(1+\left(\frac{2C_{e}}{1-\rho^{\alpha}}\right)^{(m-i({\Phi}))_+}\right)\leq \sigma.
		\end{split}
	\end{equation*} 
	In addition, by virtue of $\beta_{0}\cdot \mathfrak{b}_{k}=0$ and $g(0)=0$, we have
	\begin{equation*}
		\begin{split}
			|g_{k}|&=\Abs{\frac{g\left(\rho^{k}x\right)-g(0)-\left(\beta\left(\rho^{k}x\right)-\beta_{0}\right)\cdot \mathfrak{b}_{k}}{\rho^{k\alpha}}}\\
			&\leq \frac{[g]_{C^{\alpha^{\prime}}(T_{1})}\rho^{k\alpha^{\prime}}|x|^{\alpha^{\prime}}+[\beta-\beta_{0}]_{C^{\alpha^{\prime}}(T_{1})}\rho^{k\alpha^{\prime}}|x|^{\alpha^{\prime}} |\mathfrak{b}_{k}|}{\rho^{k\alpha}}
		\end{split}
	\end{equation*} 
	Then it follows from $\alpha\leq \alpha^{\prime}$, $\|g\|_{C^{\alpha^{\prime}}(T_{1})}\leq \delta$, $|\mathfrak{b}_{k}|\leq \frac{C_{e}}{1-\rho^{\alpha}}$, and \eqref{0801} that
	\begin{equation*}
		\begin{split}
			\|g_{k}\|_{L^{\infty}(T_{\rho^{-k}})}&\leq [g]_{C^{\alpha^{\prime}}(T_{1})}+[\beta-\beta_{0}]_{C^{\alpha^{\prime}}(T_{1})}|\mathfrak{b}_{k}|\\
			&\leq [g]_{C^{\alpha^{\prime}}(T_{1})}+[\beta-\beta_{0}]_{C^{\alpha^{\prime}}(T_{1})}\frac{C_{e}}{1-\rho^{\alpha}}\\
			&\leq \delta\left(1+\frac{C_{e}}{1-\rho^{\alpha}}\right)\leq \sigma.
		\end{split}
	\end{equation*} 
	
	At this moment, the assumptions in Lemma \ref{lem11} are satisfied. Thus, there is an affine function 
	$\tilde{\ell}(x)=\tilde{\mathfrak{a}}+\tilde{\mathfrak{b}}\cdot x$ with $\tilde{\mathfrak{a}}=u_{k}(0)=0$ and $|\tilde{\mathfrak{b}}|\leq C_{e}$ such that 
	\begin{equation}\label{22rescalingback}
		\|u_{k}-\tilde{\ell}\|_{L^{\infty}\left(\left(\frac{1}{\rho^{k}}\Omega\right)\cap B_{\rho}\right)}\leq \rho^{1+\alpha},\quad \beta_{0}\cdot \tilde{\mathfrak{b}}=0.
	\end{equation}
	To proceed, we define the $(k+1)$-th approximating affine function $\ell_{k+1}$ as 
	\begin{equation*}
		\ell_{k+1}(x):=\mathfrak{a}_{k+1}+\mathfrak{b}_{k+1}\cdot x,
	\end{equation*}
	where $\mathfrak{a}_{k+1}:=\mathfrak{a}_{k}+\rho^{k(1+\alpha)}\tilde{\mathfrak{a}}=u(0)$ and   $\mathfrak{b}_{k+1}:=\mathfrak{b}_{k}+\rho^{k\alpha}\tilde{\mathfrak{b}}$.
	Scaling \eqref{22rescalingback} back, we reach that
	\begin{equation*}
		\|u-\ell_{k+1}\|_{L^{\infty}\left(\Omega\cap B_{\rho^{k+1}}\right)}=\rho^{k(1+\alpha)}\|u_{k}-\tilde{\ell}\|_{L^{\infty}\left(B_{\rho^{k}}(x)\cap \{y_{d}>\phi_{k}(y^{\prime})\}\right)}
		\leq \rho^{(k+1)(1+\alpha)}.
	\end{equation*}
	Also,
	\begin{equation*}
		\beta_{0}\cdot {\mathfrak{b}_{k+1}}=0,\quad |\mathfrak{b}_{k+1}-\mathfrak{b}_{k}|=\rho^{k\alpha}\abs{\tilde{\mathfrak{b}}}\leq C_{e}\rho^{k\alpha}.
	\end{equation*}
	{\bf Step 3. Convergence analysis}. From Step 2, we know that $\{\mathfrak{b}_{j}\}_{j\in\mathbb{N}}\subset \rn$ is Cauchy sequence, hence, they converge, that is
	\begin{equation*}
		\quad \mathfrak{b}_{j}\rightarrow\overline{\mathfrak{b}}
	\end{equation*}	
	with	
	\begin{align}\label{511}
		|\mathfrak{b}_{j}-\overline{\mathfrak{b}}|\leq  \frac{C\rho^{j\alpha}}{1-\rho^{\alpha}}.
	\end{align}	
	
	Finally, given any $0<r\leq \rho$, we can find $j\in \mathbb{N}$ such that $\rho^{j+1}<r\leq \rho^{j}$. Then, for $\overline{\ell}(x):=\overline{\mathfrak{a}}+\overline{\mathfrak{b}}\cdot x$ with $\overline{\mathfrak{a}}:=u(0)$, we arrive at
	\begin{align*}
		\|u-\overline{\ell}\|_{L^{\infty}(\Omega_{r})}&\leq \|u-\overline{\ell}\|_{L^{\infty}(\Omega_{\rho^{j}})}\\
		&\leq \|u-{\ell_{j}}\|_{L^{\infty}(\Omega_{\rho^{j}})}+\|\ell_{j}-\overline{\ell}\|_{L^{\infty}(\Omega_{\rho^{j}})}\\
		&\leq  \rho^{j(1+\alpha)}+|\mathfrak{a}_{j}-\overline{\mathfrak{a}}|+\rho^{j}|\mathfrak{b}_{j}-\overline{\mathfrak{b}}|\\
		&\leq \rho^{j(1+\alpha)}\left(1+\frac{C}{1-\rho^{\alpha}}\right)\\
		&\leq \frac{1}{\rho^{1+\alpha}}\left(1+\frac{C}{1-\rho^{\alpha}}\right)r^{1+\alpha}\\
		&\leq Cr^{1+\alpha}.
	\end{align*}
	This implies that $u$ is $C^{1,\alpha}$ at $0$. This completes the proof of the desired result.
\end{proof}
\subsection{Proof of Theorem \ref{thm1}}
In this section, we conplete the proof of Theorem \ref{thm1}. To achieve this, we discuss the scaling properties and small regime of the equation that allow us to reduce the proof of Theorem \ref{thm1} to the hypotheses of Proposition \ref{prop3.4}.
\begin{proof}[Proof of Theorem \ref{thm1}]
	Initially, by a suitable rotation and translation of the coordinates, we may assume that $0 \in \partial\Omega$ and the domain $\Omega_1$ is given by $\{x \in B_1 : x_d > \psi_\Omega(x')\}$ for some $C^1$ function $\psi=\psi_\Omega$ with $\psi(0) = 0$ and $D\psi(0) = 0$. Let $\delta$, $\rho$ and $C_{e}$ be the constants given in Proposition \ref{prop3.4}. Next, arguing as in \cite[Section 5]{Byun2025CVPDE}, we may assume that $u$ satisfies 
\begin{equation*}
	\left\{\begin{aligned}
		\hat{\Phi}(|Du-q|,x)F(D^{2}u,x)+|Du-q|^{\vartheta}\mathcal{H}(Du-q,x) &=f|Du-q|^{\vartheta} &&\text{in } \;\;\Omega_{1}, \\
		\beta\cdot Du&=g &&\text{on }\;\; \partial\Omega_{1},
	\end{aligned}\right.
\end{equation*}
for some vector $q \in \mathbb{R}^d$, and $\nu_0=\nu_1=1$. Additionally, we may assume $u(0)=0$, $\norm{u}_{L^\infty(\Omega_{1})} \leq 1$, $g(0)=0$, and the smallness condition \eqref{081} hold
and $$|q|\leq\frac{{C_{e}}}{1-\rho^{{\alpha}}}.$$
In fact, let $u$ be a viscosity solution to 
	\begin{equation*}
	\left\{\begin{aligned}
		\hat{\Phi}(\abs{Du}, x)F(D^2 u, x)+|Du|^{\vartheta}\mathcal{H}(Du,x)&=f|Du|^{\vartheta} &&\text{in } \;\;\Omega_{1},\\
		\beta\cdot Du&={g} &&\text{on }\;\; \partial\Omega_{1}.
	\end{aligned}\right.
\end{equation*}
Let us define 
$\tilde{u}$ by
$$\tilde{u}(x)=\frac{u(\tau x)}{K}$$
for constants $K> 1>\tau>0$ to be determined later. It is
immediate to verify that $\tilde{u}$ is a viscosity solution of
\begin{equation*}
	\left\{\begin{aligned}
		\tilde{\hat{\Phi}}(\Abs{D\tilde{u}}, x)	\tilde{F}(D^2 \tilde{u}, x)+	|D\tilde{u}|^{\vartheta}\tilde{\mathcal{H}}({D\tilde{u}}, x)&= 	\tilde{f}|D\tilde{u}|^{\vartheta} &&\text{in } \;\;\left(\frac{1}{\tau}\Omega\right)\cap B_{1}, \\
		\tilde{\beta}\cdot D\tilde{u}&=g_{1} &&\text{on }\;\; \partial\left(\frac{1}{\tau}\Omega\right)\cap B_{1},
	\end{aligned}\right.
\end{equation*}
where 
\begin{align*}
	\tilde{F}(X,x):=&\frac{\tau^{2}}{K}F\left(\frac{K}{\tau^{2}}X,\tau x\right),\quad
	\tilde{\hat{\Phi}}(t,x):=\frac{\Phi\left(\frac{K}{\tau}t,\tau x\right)}{\Phi\left(\frac{K}{\tau},\tau x\right)},\\
	\tilde{\mathcal{H}}(t,x):=&\frac{\tau^{2-\vartheta}}{K^{1-\vartheta}\Phi\left(\frac{K}{\tau},\tau x\right)}\mathcal{H}\left(\frac{K}{\tau}t,\tau x\right),\quad
	\tilde{f}(x):=\frac{\tau^{2-\vartheta}}{K^{1-\vartheta}\Phi\left(\frac{K}{\tau},\tau x\right)}f(\tau x),\\
	\tilde{\beta}(x):=&\beta(\tau x),\quad
	g_{1}(x):=\frac{\tau}{K}g(\tau x).
\end{align*}
Note that $\tilde{F}$ is still a uniformly $(\lambda,\Lambda)$-elliptic operator, $\tilde{\hat{\Phi}}$ satisfies \eqref{A3} with the same constants $(i(\hat{\Phi}),s(\hat{\Phi}))$, and $\tilde{\hat{\Phi}}(1,\cdot)=1$.
A direct calculation yields 
\begin{equation*}
	{\rm osc}_{\tilde{F}}(x)=\sup\limits_{M\in S^{d}\setminus \{0\}}\frac{\Abs{F\left(\frac{K}{\tau^{2}}M, \tau x\right)-F\left(\frac{K}{\tau^{2}}M,0\right)}}{\frac{K}{\tau^{2}}\|M\|}={\rm osc}_{{F}}(\tau x).
\end{equation*}
	which immediately leads to 
\begin{equation*}
	\begin{split}
		\left(\fint_{\left(\frac{1}{\tau}\Omega\right)\cap B_{1/2}}\Abs{{\rm osc}_{\tilde{F}}(x)}^{d}\text{d}x\right)^{1/d}=\left(\int_{\Omega_{1/2}}\Abs{{\rm osc}_{{F}}(x)}^{d}\text{d}x\right)^{1/d}\leq \eta_{0}\leq \delta.
	\end{split}
\end{equation*} 
provided $\eta_{0}$ is chosen sufficiently small. Since $\psi_{\Omega}\in C^1$, there exists $0<\tilde{\tau}\ll1$, depending only on the $C^1$ modulus
of $\psi_{\Omega}$, such that
\[
|D\psi_{\Omega}(x)|\leq \delta
\qquad\text{for all } |x|\le \tilde{\tau}.
\]
Set
\[
\tilde{\psi}(x):=\psi_{\frac{1}{\tilde{\tau}}\Omega}(x)=\frac{\psi({\tau}x)}{{\tau}}
\]
for $\tau\leq \tilde{\tau}$. Then, we choose $\tau\in(0,1)$ such that $\|\tilde{\psi}\|_{C^1(\mathrm{T}_{\tau^{-1}})}\leq \delta$ and
\[
[\tilde{\beta}-\beta_0]_{C^{\alpha^{\prime}}(\mathrm{T}_{\tau^{-1}})}
\le \tau^{\alpha^{\prime}}[\beta-\beta_0]_{C^{\alpha^{\prime}}(\mathrm{T}_{1})}\leq \frac{1-\rho^{\alpha}}{2C_{e}}\delta.
\]
Combining \eqref{A2} with \eqref{A3} and $\frac{K}{\tau}\geq 1$, we get 
\begin{equation*}
	|\tilde{\mathcal{H}}(t,x)|\leq \frac{L\tau^{2+i(\hat{\Phi})-\vartheta}}{\nu_{0}K^{1+i(\hat{\Phi})-\vartheta}}\bigg(\mathcal{K}+\mathcal{M}\left(\frac{K}{\tau}\right)^{m}|t|^{m}\bigg)=:\tilde{\mathcal{K}}+\tilde{\mathcal{M}}|t|^{m}, 
\end{equation*}
\begin{equation*}
	\|\tilde{f}\|_{L^{\infty}\left(\frac{1}{\tau}\Omega\cap B_{1}\right)}\leq \frac{L\tau^{2+i(\hat{\Phi})-\vartheta}}{\nu_{0}K^{1+i(\hat{\Phi})-\vartheta}}\|{f}\|_{L^{\infty}(\Omega_{1})}.
\end{equation*}
Next, for $0<m<1+i(\Phi)$, we select
$$K:=1+4\|u\|_{L^{\infty}(\Omega)}+\frac{2\|g\|_{C^{\alpha^{\prime}}(T_{1})}}{\mu_{0}\delta}+\left(\frac{L\left(\|f\|_{L^{\infty}(\Omega)}+\mathcal{K}\right)}{\nu_{0}\delta}\right)^{\frac{1}{1+i(\hat{\Phi})-\vartheta}}+\left(\frac{L\mathcal{M}}{\nu_{0}\delta}\right)^{\frac{1}{1+i(\hat{\Phi})-m-\vartheta}}.$$
For $m=1+i(\Phi)$, we choose $\tau\in(0,1)$ fulfilling the conditions stated above and, in addition, such that 
\[
\tau\leq \frac{\nu_{0}\delta}{L\mathcal{M}},
\] 
and 
$$K:=1+4\|u\|_{L^{\infty}(\Omega)}+\frac{2\|g\|_{C^{\alpha^{\prime}}(T_{1})}}{\mu_{0}\delta}+\left(\frac{L\left(\|f\|_{L^{\infty}(\Omega)}+\mathcal{K}\right)}{\nu_{0}\delta}\right)^{\frac{1}{1+i(\hat{\Phi})-\vartheta}}.$$
With these choices, we see that
$$\|\tilde{u}\|_{L^{\infty}\left(\frac{1}{\tau}\Omega\cap B_{1}\right)}\leq \frac{1}{4},\quad \|g_{1}\|_{C^{\alpha^{\prime}}(T_{\tau^{-1}})}\leq \frac{\mu_{0}}{2}\delta ,$$
$$\|\tilde{f}\|_{L^{\infty}\left(\frac{1}{\tau}\Omega\cap B_{1}\right)}+\tilde{\mathcal{K}}\leq \delta,\quad \tilde{\mathcal{M}}\leq \delta.$$
To proceed, denote
$$\tilde{v}(x):= \tilde{u}(x)-\frac{{g}_{1}(0)}{\tilde{\beta}_d(0)} x_d -\tilde{u}(0)$$
and $$\tilde{q}:=-\frac{g_1(0)}{\tilde{\beta}_d(0)} e_d$$
Then $\tilde{v}$ is a viscosity solution of 
 \begin{equation*}
 	\left\{\begin{aligned}
 		\tilde{\hat{\Phi}}(\Abs{D\tilde{v}-\tilde{q}}, x)	\tilde{F}(D^2 \tilde{v}, x)+	|D\tilde{v}-\tilde{q}|^{\vartheta}\tilde{\mathcal{H}}({D\tilde{v}-\tilde{q}}, x)&= 	\tilde{f}|D\tilde{v}-\tilde{q}|^{\vartheta} &&\text{in } \;\;\left(\frac{1}{\tau}\Omega\right)\cap B_{1}, \\
 		\tilde{\beta}\cdot D\tilde{v}&=\tilde{g} &&\text{on }\;\; \partial\left(\frac{1}{\tau}\Omega\right)\cap B_{1},
 	\end{aligned}\right.
 \end{equation*}
where $\tilde{g}=g_1-\frac{g_1(0)}{\tilde{\beta}_n(0)} \tilde{\beta}_n$. It follows from $\beta_{d}(0)\geq \mu_{0}$ and $\delta\leq \frac{{C_{e}}}{1-\rho^{{\alpha}}}$ that
$$|\tilde{q}|\leq \Abs{\frac{g_1(0)}{\tilde{\beta}_d(0)}}\leq\frac{|g_1(0)|}{\mu_{0}}\leq \frac{\delta}{2}\leq \frac{{C_{e}}}{1-\rho^{{\alpha}}}.$$
Note that
$\tilde{v}(0)=0$, $\tilde{g}(0)=0$, and
\[
\|\tilde{v}\|_{L^{\infty}\left(\frac{1}{\tau}\Omega\cap B_{1}\right)}
\le 2 \|\tilde{u}\|_{L^{\infty}\left(\frac{1}{\tau}\Omega\cap B_{1}\right)} +\frac{|g_1(0)|}{\mu_0}\leq \frac{1}{2}+\frac{\delta}{2}
\le 1.
\]
Moreover, we have
\begin{equation}
	\|\tilde{g}\|_{C^{\alpha^{\prime}}(\mathrm{T}_{\tau^{-1}})}\leq
	\|g_1\|_{C^{\alpha^{\prime}}(\mathrm{T}_{\tau^{-1}})}
	+\frac{|g_1(0)|}{\mu_0}\|\tilde{\beta}_d\|_{C^{\alpha^{\prime}}(\mathrm{T}_{\tau^{-1}})}\leq \frac{\mu_{0}}{2}\delta +\frac{{C_{e}}}{1-\rho^{{\alpha}}}\frac{1-\rho^{\alpha}}{2C_{e}}\delta
	\leq \delta.
\end{equation}	
	
	At this point, the assumptions in Proposition \ref{prop3.4} have been fulfilled, and hence we obtain that $u\in C^{1,\alpha}(0)$.  Finally, by combining the interior $C^{1,\alpha_{1}}$ regularity \cite[Theorem 1.1]{Huo2026} and boundary regularity with standard covering arguments, we can obtain the global $C^{1,\alpha}$ regularity with $\alpha$ satisfying \eqref{exponent}. Thus, the proof of Theorem \ref{thm1} for the case $\gamma=0$ is complete.
	
	For the general case, let $u$ be a viscosity solution to \eqref{8866}. Then, $u$ also solves the following problem:
		\begin{equation}\label{8868}
		\left\{\begin{aligned}
			\hat{\Phi}(\abs{Du}, x)F(D^2 u, x)+|Du|^{\vartheta}\mathcal{H}(Du,x)&=f|Du|^{\vartheta} &&\text{in } \;\;\Omega_{1},\\
			\beta\cdot Du&=\hat{g} &&\text{on }\;\; \partial\Omega_{1},
		\end{aligned}\right.
	\end{equation}
	where $\hat{g}=g-\gamma u$. By Lemma \ref{holder}, we know  $u\in C^{0,\alpha^{\prime}}$ up to the boundary, which implies that $\hat{g}\in C^{0,\alpha^{\prime}}(\partial\Omega_{1})$. Consequently, we can apply Proposition \ref{prop3.4} to the problem \eqref{8868} and obtain the desired regularity result. The proof is complete.
\end{proof}


\section*{Data availability} Data sharing is not applicable to this article as obviously no datasets were generated or analyzed during the current study.

\section*{Conflict of interest} Author states no conflict of interest.

\end{document}